\documentclass[]{amsart}

\usepackage{amssymb}           
\usepackage{bm}                

\usepackage{amsmath}           
\usepackage{amsthm}            
\usepackage{amscd}             

\usepackage{ascmac}            
\usepackage{indentfirst}       

\usepackage[pdftex]{graphicx} 
\newtheorem{theorem}{Theorem}[section]
\newtheorem{prop}[theorem]{Proposition}
\newtheorem{cor}[theorem]{Corollary}
\newtheorem{lem}[theorem]{Lemma}

\theoremstyle{definition}

\begin{document}


%
%
\renewcommand{\subjclassname}{
\textup{2020} Mathematics Subject Classification}

\date{\today}


\title[Triangles and quadrilaterals]{
Any link has a diagram with only triangles and quadrilaterals
}


\author{Reiko Shinjo}
\address{School of Science and Engineering, 
Kokushikan University, 4-28-1, Setagaya, Setagaya-ku, Tokyo 154-8515, Japan}
\email{reiko@kokushikan.ac.jp}

\author{Kokoro Tanaka}
\address{Department of Mathematics, Tokyo Gakugei University, 
Nukuikita 4-1-1, Koganei, Tokyo 184-8501, Japan}
\email{kotanaka@u-gakugei.ac.jp}


\dedicatory{Dedicated to Professor Kouki Taniyama on the occasion of his 60th birthday}

\begin{abstract}
A link diagram can be considered as a $4$-valent graph embedded in the $2$-sphere 
and divides the sphere into complementary regions. 
In this paper, we show that any link has a diagram with only triangles and quadrilaterals. 
This extends previous results shown by the authors and C.~Adams. 
%
\end{abstract}

\maketitle



\section{Introduction}\label{sec:intro}

A link diagram can be considered as a $4$-valent graph embedded in the $2$-sphere 
and divides the sphere into complementary regions. 
Given a reduced connected diagram $D$, 
let 
$p_n$ 
(or more precisely $p_n(D)$)    
be the number of $n$-gons of all the complementary regions of $D$ for each $n \geq 1$. 
Note that 
$p_1 = 0$
since $D$ is reduced. 
It follows from Euler's formula and some elementary observations that 
we have the following equation
\begin{equation}\label{eq:Euler}
2 p_2 + p_3 = 8 + p_5 + 2 p_6 + 3 p_7 + \cdots ,
\end{equation}
in which 
$p_4$
does not appear; see for example \cite{AST}. 

In graph theory, the converse direction has been investigated, dating back to \cite{Eberhard}.
Gr\"{u}nbaum \cite{Grunbaum} proved that any sequence $\{p_n\}_{n \geq 2, n \neq 4}$ 
of nonnegative integers with $p_2=0$ that satisfies Equation~\eqref{eq:Euler} can be realized 
as a planar $4$-valent $3$-connected graph such that  
the number of its $n$-gon regions is $p_n$ for all $n \neq 4$. 
This theorem is known as Eberhard's theorem. 
We note that the condition $p_2=0$ is a typical assumption in graph theory, 
since the $1$-skeletons of convex polytopes are of interest.  
We also note that a convex $3$-polytopal $4$-valent graph is nothing but 
a connected reduced link projection without bigons. 
Then Jeong \cite{Jeong} extended this result to show that the resulting graph can be taken as 
a knot projection rather than a link projection. 

In this paper, we investigate such a problem for link diagrams; 
which sequence $\{p_n\}_{n \geq 2, n \neq 4}$ 
of nonnegative integers that satisfies Equation~\eqref{eq:Euler} can be realized 
as a diagram of every link such that the number of its $n$-gon regions is $p_n$ 
for all $n \neq 4$? 
This is a continuation of the study of complementary regions of knot and link diagrams 
by the authors and Adams in \cite{AST}. 
In that paper, we introduced and investigated the notion of universal sequences for knots and links, 
which will be reviewed and discussed in Section~\ref{sec:universal} of this paper.  
See the original paper \cite{AST} or the survey \cite[Chapter 10]{Adams} for more details 
about universal sequences. 
The main purpose of this paper is to show the following theorem and its corollary. 

\begin{theorem}\label{thm:main}
For any sequence $\{ p_n \}_{n \geq 2, n \neq 4}$ of nonnegative integers 
with $p_2 = 0$ that satisfies Equation~\eqref{eq:Euler} 
and any link $L$, there exists 
a diagram $D_L$ of $L$ such that $p_n(D_L) = p_n$ for all $n \neq 4$.  
%
\end{theorem}

Since the sequence $p_2=0$, $p_3=8$ and $p_n=0$ ($n \geq 5$) 
satisfies the assumption of 
Theorem~\ref{thm:main}, 
we have the following, which is the title of this paper: 

\begin{cor}\label{cor:(3,4)} 
Any link has a diagram 
with only eight triangles and quadrilaterals. 
\end{cor}


Theorem~\ref{thm:main} follows from Theorem~\ref{thm:template},  
whose proof will be given in Section~\ref{sec:template}. 
 
\begin{theorem}\label{thm:template}
Let $P$ be a knot projection with the part as shown in Figure~\ref{fig:2}, 
where the numbers from $0$ to $2$ indicate the order in which the arcs are traced. 
Then any link $L$ has a diagram $D_L$ such that 
$p_n(D_L) = p_n(P)$ for all $n \neq 4$. 
\end{theorem}

\begin{figure}[thbp]
\includegraphics[width=0.40\textwidth, pagebox=artbox]{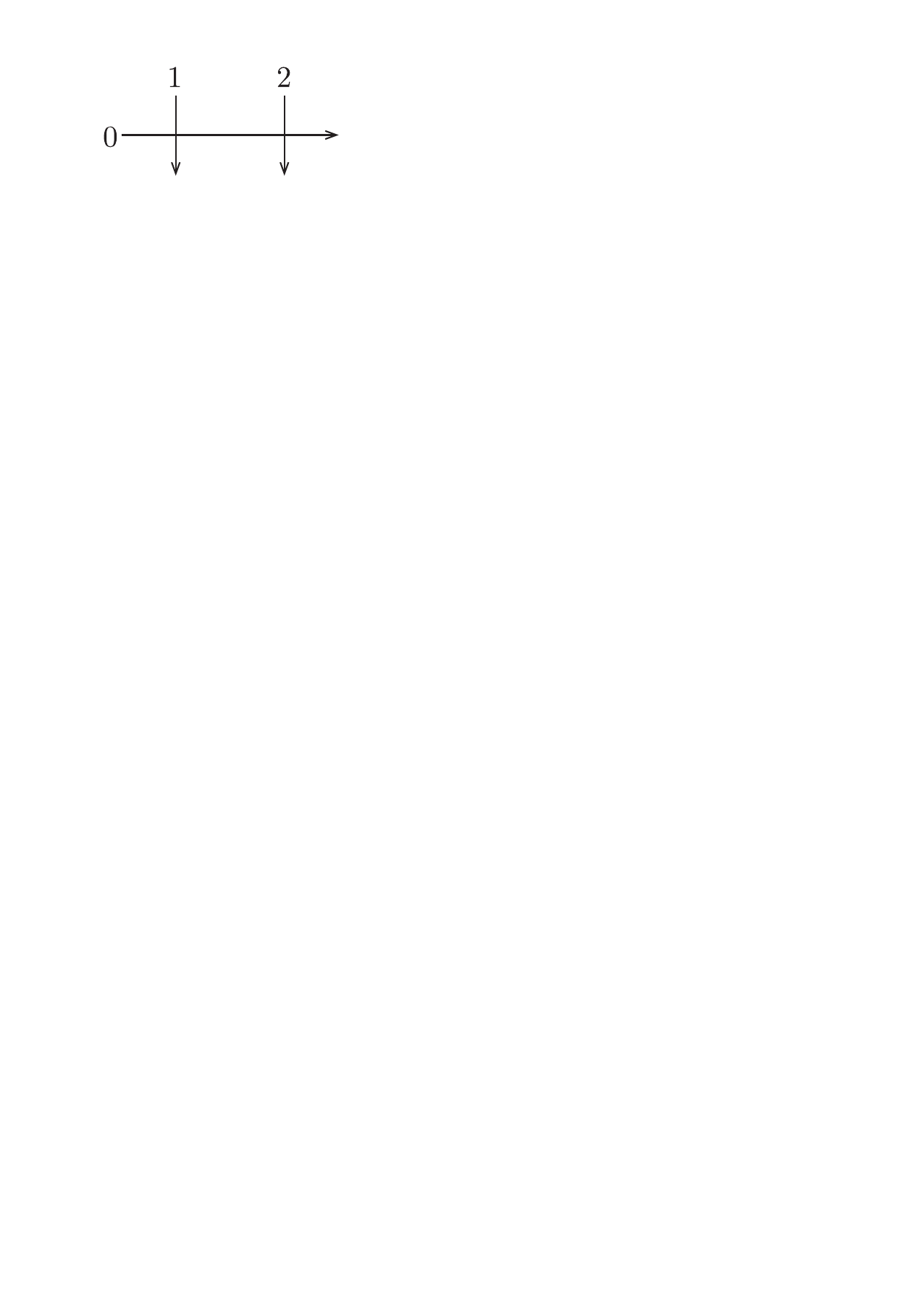}
\caption{A part of a knot projection consisting of three strands}\label{fig:2}
\end{figure}

\begin{proof}[Proof of Theorem~\ref{thm:main}]
Let $\{ p_n \}_{n \geq 2, n \neq 4}$ be a sequence of 
nonnegative integers with $p_2 = 0$ that satisfies Equation~\eqref{eq:Euler}.  
It was shown in \cite{Jeong} that 
there exists a choice of $p_4$ and a knot projection $P$ such that 
$p_n(P) = p_n$ for all $n \geq 2$. 
The proof was inductive and constructive. 
We recall here a rough outline of the proof in \cite{Jeong}.   
Start with the knot projection $P_0$ in Figure~\ref{fig:Jeong}. 
It is made up of eight triangles and three quadrilaterals, and 
has the part as shown in Figure~\ref{fig:2}; see the right of Figure~\ref{fig:Jeong}. 
By performing some local operations for $P_0$ repeatedly outside the part as shown in Figure~\ref{fig:2}, 
the knot projection $P_0$ can be changed into the desired knot projection $P$, which 
also has the the part as shown in Figure~\ref{fig:2}.  
Thus Theorem~\ref{thm:template} can be applied for $P$ 
and hence any link $L$ has a diagram $D_L$ such that $p_n(D_L) = p_n(P) = p_n$ for all $n \neq 4$. 
\end{proof}

\begin{figure}[thbp]
\includegraphics[width=0.85\textwidth, pagebox=artbox]{figure/2.pdf}
\caption{Jeong's knot projection $P_0$}\label{fig:Jeong}
\end{figure}

Although some readers may wonder if it is possible to remove the assumption 
$p_2=0$ from Theorem~\ref{thm:main}, we think that it is an issue in the future. 
This is because, in order to use Theorem~\ref{thm:template}, 
we have to extend the result by Jeong \cite{Jeong} without the assumption $p_2=0$, 
however, it may be difficult at this time.  
Instead, for example, it is possible to prove the following 
for the sequence $p_2=2, p_3=4$ and $p_n=0$ ($n \geq 5$) 
which satisfies Equation~\eqref{eq:Euler}. 

\begin{prop}
Any link has a diagram 
with only two bigons, four triangles and quadrilaterals. 
\end{prop}

\begin{proof}
Take a knot projection $P_1$ as in the left of Figure~\ref{fig:(2,4,odd)} 
so that 
$p_2(P_1) = 2$, $p_3(P_1)=4$ and $p_n(P_1) = 0 $ for all $n \geq 4$. 
Since  $P_1$ has the part as shown in Figure~\ref{fig:2},  
Theorem~\ref{thm:template} can be applied for the projection $P_1$, 
and hence any link $L$ has a diagram $D_L$ such that $p_n(D_L) = p_n(P_1)$ for all $n \neq 4$, 
that is, $p_2(D_L) = 2$, $p_3(D_L)=4$ and $p_n(D_L) = 0 $ for all $n \geq 5$. 
\end{proof}

Throughout this paper, we use the fact that any diagram can be made a diagram 
of the unknot by crossing changes. 
Equivalently, we can make any knot projection into a diagram of the unknot 
by giving crossing information appropriately. 
In fact, we can say a little stronger assertion, which will be used later.

\begin{lem}\label{lem:writhe}
Let $Q$ be a knot projection. 
If the number of crossings of $Q$ is even $($resp. odd$)$, 
then there exists a diagram $D_Q$ of the unknot whose underlying projection is $Q$ 
and whose writhe is $0$ $($resp. $+1$$)$. 
\end{lem}

\begin{proof}
We apply induction on the number of crossings.
\end{proof}

\section{Proof of Theorem~\ref{thm:template}}\label{sec:template}

\begin{lem}\label{lem:2toN}
Let $P$ be a knot projection with the part as shown in Figure~\ref{fig:2}, 
where the numbers from $0$ to $2$ in the figure 
indicate the order in which the arcs are traced. 
Then there exists a knot projection $Q$ for any integer $N \geq 3$ 
with the part as shown in Figure~\ref{fig:N} 
such that $p_n(Q) = p_n(P)$ for any $n \neq 4$, 
where the numbers from $0$ to $N$ in the figure 
also indicate the order in which the arcs are traced. 
\end{lem}

\begin{figure}[thbp]
\includegraphics[width=0.50\textwidth, pagebox=artbox]{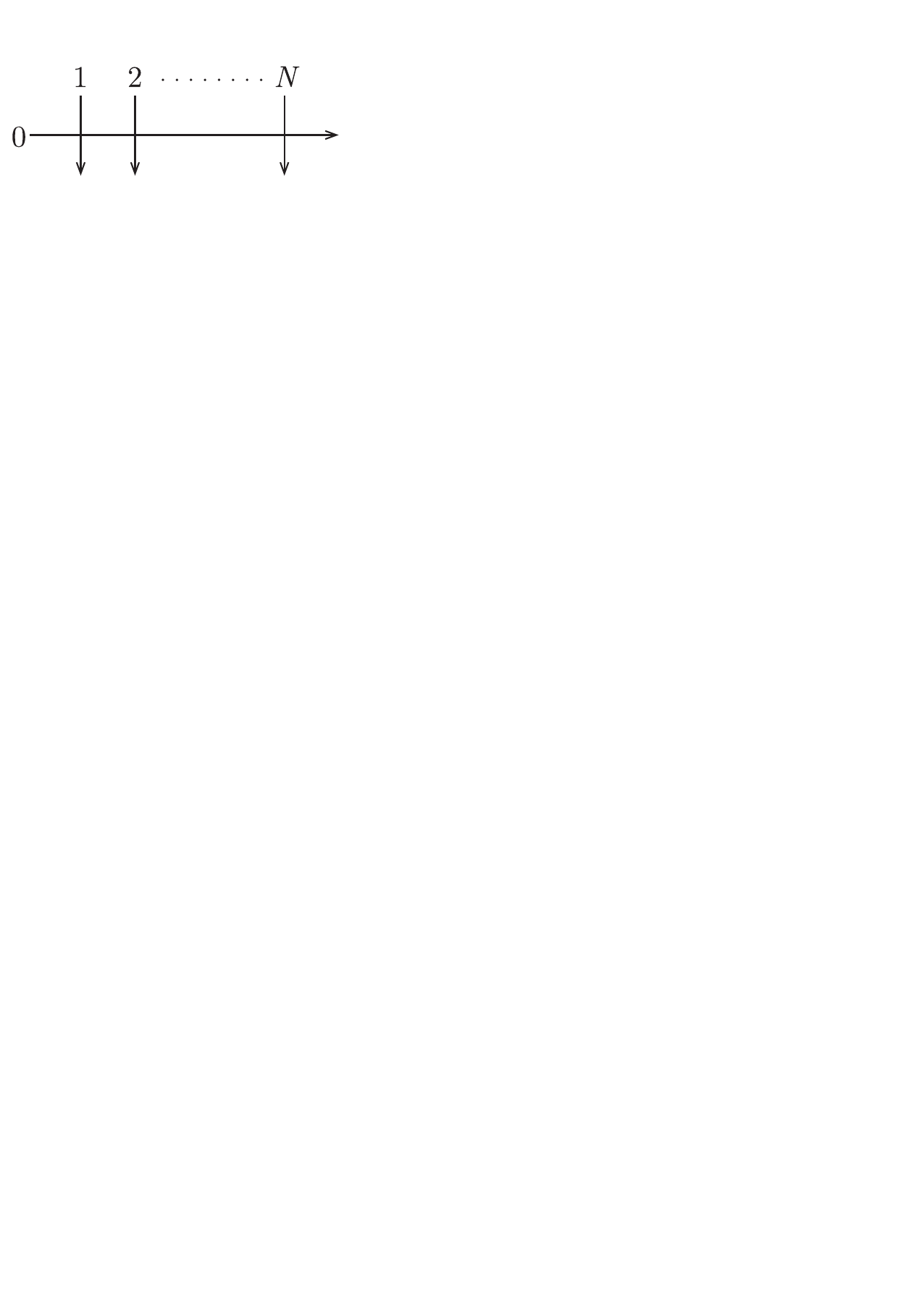}
\caption{A part of a knot projection consisting $(N+1)$-strands}\label{fig:N}
\end{figure}

\begin{proof}
Let $Q_1$ be an $N$-parallel copies of $P$ so that $p_n(Q_1) = p_n(P)$ for all $n \neq 4$. 
%
Let $Q_2$ be a knot projection obtained from $Q_1$ by replacing a part of $Q_1$, 
which corresponds to the $N$-parallel copies of Figure~\ref{fig:N},  
with the part as in the top of Figure~\ref{fig:2toN}, where it depicts the case $N=3$. 
When we ignore quadrilaterals, 
regions of $Q_2$ are almost identical to those of $P$. 
The exact differences are as follows: 
each of the two shaded regions of $Q_2$ in the top of Figure~\ref{fig:2toN} has one more edge than 
the corresponding region in $P$, and has a newly created triangle next to it. 
Roughly speaking, 
let $Q$ be a knot projection obtained from $Q_2$ by creating a pair of two kinks 
as in the middle of Figure~\ref{fig:2toN}, 
and aligning them along with $Q_2$ as in the bottom of Figure~\ref{fig:2toN}. 
The precise procedure making $Q$ from $Q_2$ is as follows.  

Create a pair of two kinks as in the middle of Figure~\ref{fig:2toN} 
such that 
each kink is obtained by a Reidemeister move of type I (without crossing information) 
along the edge between a shaded region and the triangle next to it, 
where the newly created $1$-gon is adjacent to the shaded region. 
Align each of the two kinks along with $Q_2$ as in the bottom of Figure~\ref{fig:2toN} 
such that 
the upper kink goes out from $0$ and returns back into $1$, and that 
the lower kink goes out from $2$ and returns back into $0$, 
where each number from $0$ to $2$ indicates a set of parallel strands of $Q_2$ 
corresponding to the strand labelled by the number in Figure~\ref{fig:N}. 

Then we actually obtain the new knot projection $Q$ from $Q_2$. 
%
%
%
It follows from the construction that 
$p_n(Q) = p_n(P)$ for all $n \neq 4$. 
Moreover the knot projection $Q$ has the part as in Figure~\ref{fig:N}; 
see the bottom right of Figure~\ref{fig:2toN}. 
Hence the knot projection $Q$ is a desired one.
\end{proof}

\begin{figure}[thbp]
\includegraphics[width=1.00\textwidth, pagebox=artbox]{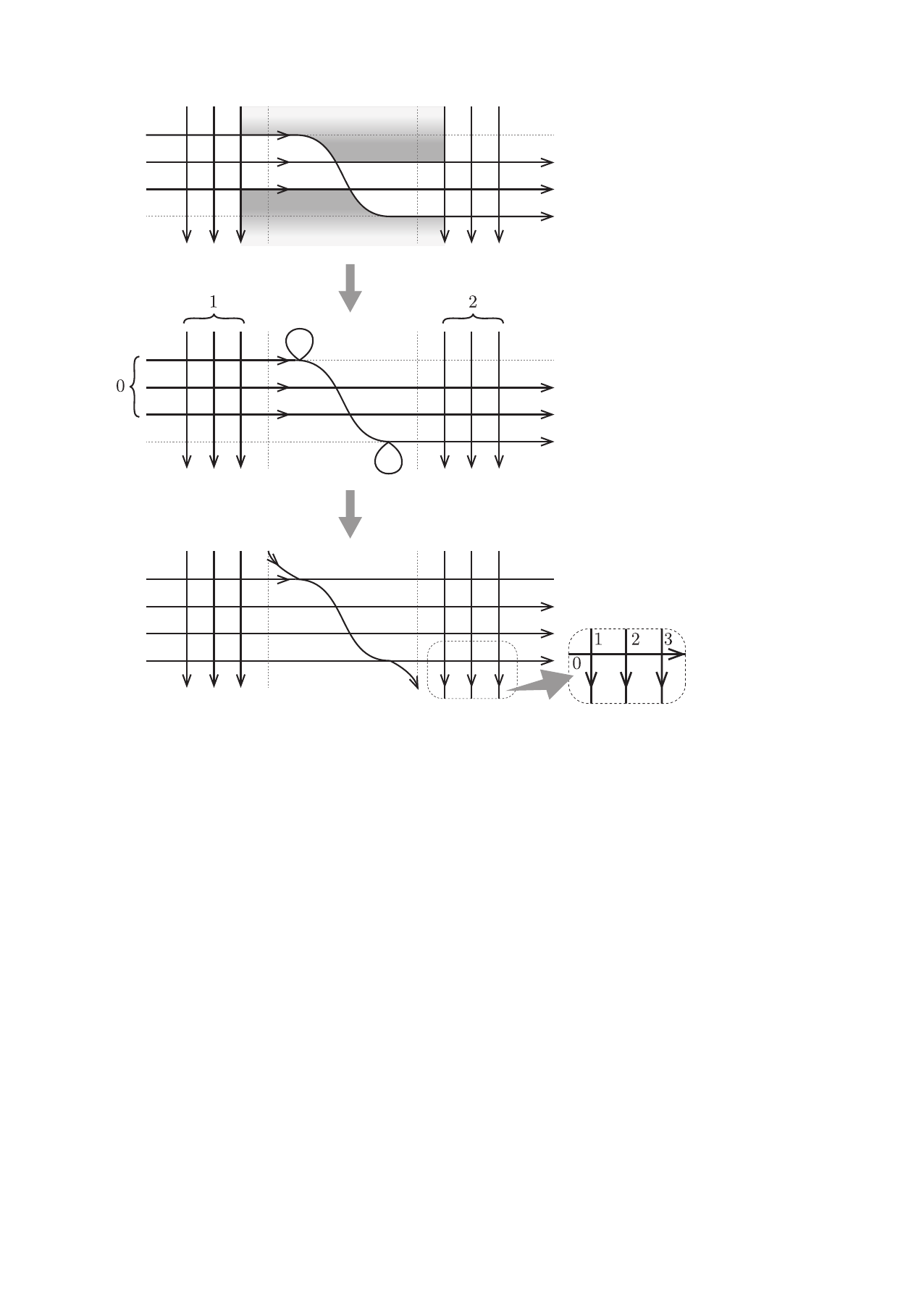}
\caption{How to make $Q$ from $P$ for the case $N=3$}\label{fig:2toN}
\end{figure}

\begin{prop}\label{prop:template}
Let $L$ be a link, and $Q$ 
a knot projection with the part as shown in Figure~\ref{fig:N} 
for sufficiently large $N$ depending on $L$, 
where the numbers from $0$ to $N$ in the figure 
indicate the order in which the arcs are traced. 
Then $L$ has a diagram $D_L$ such that 
$p_n(D_L) = p_n(Q)$ for any $n \neq 4$. 
\end{prop}

\begin{proof}
Take a closed quasitoric braid diagram $D_1$ of the link $L$, 
where a quasitoric braid is a braid obtained 
by changing some subset of the crossings in a toric braid. 
We note that every link can be realized as the closure of a quasitoric braid \cite{Manturov}. 
When the quasitoric braid diagram $D_1$ for $L$ is of type $(p,q)$, 
we take $N$ so that $N > p+q$ 
and a knot projection $Q$ as in the statement. 
We discuss the cases where the number of crossings of $Q$ is odd and even. \\

\underline{Case $1$.}\ 
Consider the case where the number of crossings of $Q$ is even. 
It follows from Lemma~\ref{lem:writhe} that 
there exists a diagram $D_Q$ of the unknot whose underlying projection is $Q$  
and whose writhe is $0$. 
Align the closed quasitoric braid diagram $D_1$ along with 
the diagram $D_Q$ of the unknot such that quasitoric braid parts are arranged 
as in the top of Figure~\ref{fig:qtb}, where it depicts the case $(p,q)=(3,2)$, 
and denote the diagram obtained from $D_1$ by $D_2$. 
Since the writhe of $D_Q$ is zero, the diagram $D_2$ represents the link $L$. 
When we ignore quadrilaterals, 
regions of $D_2$ are almost identical to those of $Q$. 
The exact differences are as follows: 
each of $2q$ shaded regions of $D_2$ in the top of Figure~\ref{fig:qtb} has one more edge than 
the corresponding region of $Q$, and has a newly created triangle next to it. 
%
Roughly speaking, 
let $D_L$ be a diagram obtained from $D_2$ by creating $2q$ kinks 
as in the middle of Figure~\ref{fig:qtb}, 
and aligning them along with $Q_2$ as in the bottom of Figure~\ref{fig:qtb}. 
During this process, 
we can trace the part of the link with the kink by placing the kink strand 
on top of the strand of the link, the knot type remains the same, 
and the crossing information for the kink is irrelevant.
The precise procedure making $D_L$ from $D_2$ is as follows.  

Create $2q$ kinks for $D_2$ as in the middle of Figure~\ref{fig:qtb} 
such that each kink is obtained by a Reidemeister move of type I along the edge between 
a shaded region and the triangle next to it, 
where the newly created $1$-gon is adjacent to the shaded region.  
Divide the $2q$ kinks into $q$ pairs of two kinks as in the middle of Figure~\ref{fig:qtb} 
such that  
%
\begin{itemize}
\item 
the first pair consists of 
a kink adjacent to the first shaded region from the right on the upper and 
that adjacent to the first shaded region from the left on the lower, 
\item 
the second pair consists of 
a kink adjacent to the second shaded region from the right on the upper and 
that adjacent to the second shaded region from the left on the lower, 
\item[] \quad $\vdots$
\item
the $(q-1)$-st pair consists of 
a kink adjacent to the $(q-1)$-st shaded region from the right on the upper and 
that adjacent to the $(q-1)$-st shaded region from the left on the lower, 
\item[] \quad and 
\item
the $q$-th pair consists of 
a kink adjacent to the $q$-th shaded region from the right on the upper and 
that adjacent to the $q$-th shaded region from the left on the lower. 
\end{itemize}
%
Align each of the $2q$ kinks 
along with $D_Q$ as in the bottom of Figure~\ref{fig:qtb}  
according to the order such that 
%
\begin{itemize}
\item 
%
the upper kink of the first pair 
leaves from $0$, go through from $1$ to $q-1$, and returns to $q$, 
and 
%
the lower kink of the first pair 
leaves from $2$, go through from $3$ to $q+1$, and returns to $0$, 
\item 
%
the upper kink of the second pair 
leaves from $0$, go through from $1$ to $q-2$, and returns to $q-1$,  
and 
%
the lower kink of the second pair 
leaves from $3$, go through from $4$ to $q+1$, and returns to $0$, 
\item[] \quad $\vdots$
\item
%
the upper link of the $(q-1)$-st pair
leaves from $0$, go through $1$, and returns to $2$, 
and  
%
the lower link of the $(q-1)$-st pair
leaves from $q$, go through $q+1$, and returns to $0$,  
\item[] \quad and 
\item
%
the upper link of the $q$-th pair
leaves from $0$ and returns to $1$, 
and  
%
the lower link of the $q$-th pair
leaves from $q+1$ and returns to $0$, 
\end{itemize}
%
where each number from $0$ to $q+1$ indicates a set of parallel strands of $D_2$ 
corresponding to the strand labelled by the number in Figure~\ref{fig:N}. 

Then we actually obtain the new diagram $D_L$ from $D_2$, where 
crossing information concerning aligned kinks of $D_L$ can be suitably chosen 
such that $D_L$ represents the link $L$, in other words, 
such that 
all aligned kinks of $D_L$ can be shrunk back to the original positions in $D_2$. 
%
It follows from the construction that 
$p_n(D_L) = p_n(Q)$ for all $n \neq 4$, 
and hence the diagram $D_L$ is a desired one. \\

\begin{figure}[thbp]
\includegraphics[width=0.90\textwidth, pagebox=artbox]{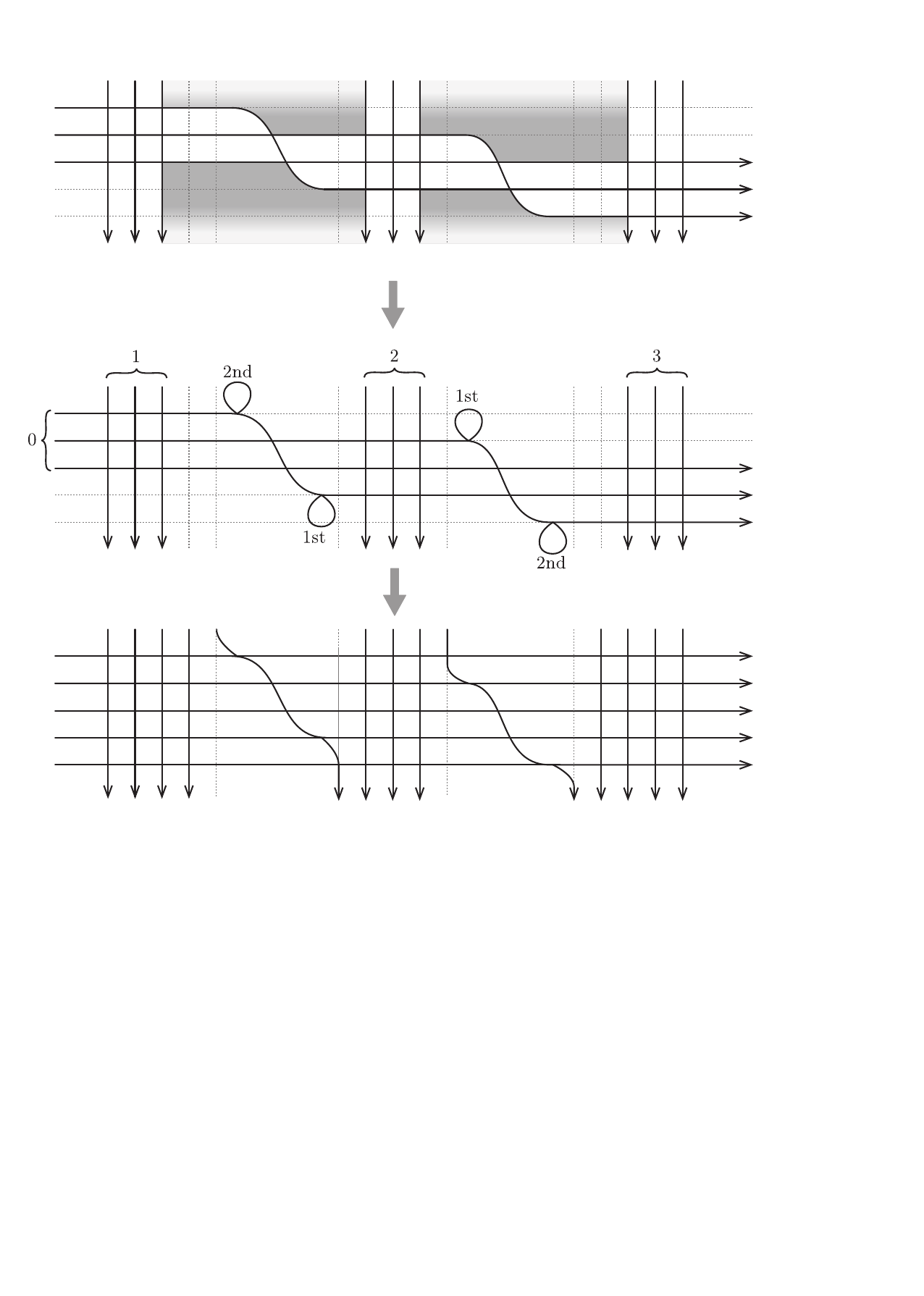}
\caption{How to make $D$ from $D_1$ for the case $(p,q)=(3,2)$}\label{fig:qtb}
\end{figure}

\underline{Case $2$.}
Consider the case where the number of crossings of $Q$ is odd. 
It follows from Lemma~\ref{lem:writhe} that 
there exists a diagram $D_Q$ of the unknot whose underlying projection is $Q$  
and whose writhe is $+1$. 
Suppose that the quasitoric braid diagram $D$ of the link $L$ is the closure of a quasitoric braid $b_L$. 
Then take a closed quasitoric braid diagram $D'_1$ as the closure of 
the product of the two quasitoric braid $b_L$ and $b_{-1}=(\sigma_1 \sigma_2 \cdots \sigma_{p-1})^{-p}$. 
We note that the product of two quasitoric braid is also quaitoric \cite{Manturov}. 
We remark here that $D'_1$ does not represent $L$, since the writhe of $D_Q$ is $+1$. 
Align $D'_1$ along with $D_Q$ such that quasitoric braid parts are arranged 
as in the top of Figure~\ref{fig:qtb}, and denote the diagram obtained from $D'_1$ by $D_2$. 
Since the writhe of $D_Q$ is $+1$ and the (quasitoric) braid $b_{-1}$ represents the $(-1)$-full-twist, 
the diagram $D_2$ represents the link $L$. 
The rest of the proof is the same as for Case 1 and is therefore omitted.
\end{proof}

\begin{proof}[Proof of Theorem~\ref{thm:template}]
It directly follows from Lemma~\ref{lem:2toN} and Proposition~\ref{prop:template}. 
\end{proof}

\section{Universal sequences}\label{sec:universal}

We discuss the relationship with the notion of universal sequences \cite{AST} introduce 
by the authors and Adams.  
A strictly increasing sequence of integers $(a_1,a_2,a_3,\ldots)$ with $a_1 \geq 2$ 
is said to be \textit{realized} by a link if there exists a diagram for the link 
such that each complementary region is an $a_n$-gon for some $a_n$ that 
appears in the sequence. We note that not every $a_n$ must be realized by a region. 
We say that a sequence is \textit{universal}\footnote{
Although the term ``universal for knots and links'' is used in \cite{AST}, 
the term ``universal'' is simply used in this paper.
}\ 
if every link has a diagram 
realizing the sequence. In \cite{AST}, the following were shown: 
\begin{itemize} 
\item $(n,2n,3n,\ldots)$ is not universal for any $n \geq 2$ (\cite[Theorem 2.3]{AST}), 
\item $(3,5,7,\ldots)$ is universal (\cite[Theorem 3.1]{AST}),  
\item $(3,n,n+1,n+2,\ldots)$ is universal for any $n \geq 4$ (\cite[Theorem 3.1]{AST}),   
\item $(2,n,n+1,n+2,\ldots)$ is universal for any $n \geq 3$ (\cite[Theorem 3.1]{AST}), 
\item $(3,4,n)$ is universal for any $n \geq 5$ (\cite[Theorem 3.2 and 3.3]{AST}), and  
\item $(2,4,5)$ is universal (\cite[Theorem 3.4]{AST}).  
\end{itemize} 

Using Theorem~\ref{thm:template}, we can extend the last two results 
and thus give alternative proofs for them. 
Note that Theorem~\ref{thm:(3,4)} implies the fifth result above. 

\begin{theorem}\label{thm:(3,4)}
The sequence $(3,4)$ is universal. 
\end{theorem}

\begin{proof}
This is a paraphrased assertion of Corollary~\ref{cor:(3,4)}. 
\end{proof}

\begin{theorem}
The sequence $(2,4,2k+1)$ is universal for any $k \geq 2$. 
\end{theorem}

\begin{proof}
Take a knot projection $P_k$ as in the left of Figure~\ref{fig:(2,4,odd)} for each $k \geq 2$ 
so that 
\[ p_2(P_k) = 4k-2,\ p_{2k+1}(P_k) = 4 \ \ \text{and} \ \  p_n(P_k) = 0 \ (n \neq 2,2k+1) . \]
Since  $P_k$ has the part as shown in Figure~\ref{fig:2},  
Theorem~\ref{thm:template} can be applied for the projection $P_k$, 
and hence any link $L$ has a diagram $D_L$ such that $p_n(D_L) = p_n(P_k)$ for all $n \neq 4$. 
It follows from $p_n(P_k) = 0 $ for all $n \neq 2,2k+1$ that $p_n(D_L) = 0$ for all $n \neq 2,4,2k+1$. 
This implies that the sequence $(2,4,2k+1)$ is universal.  
\end{proof}

\begin{figure}[thbp]
\includegraphics[width=0.95\textwidth, pagebox=artbox]{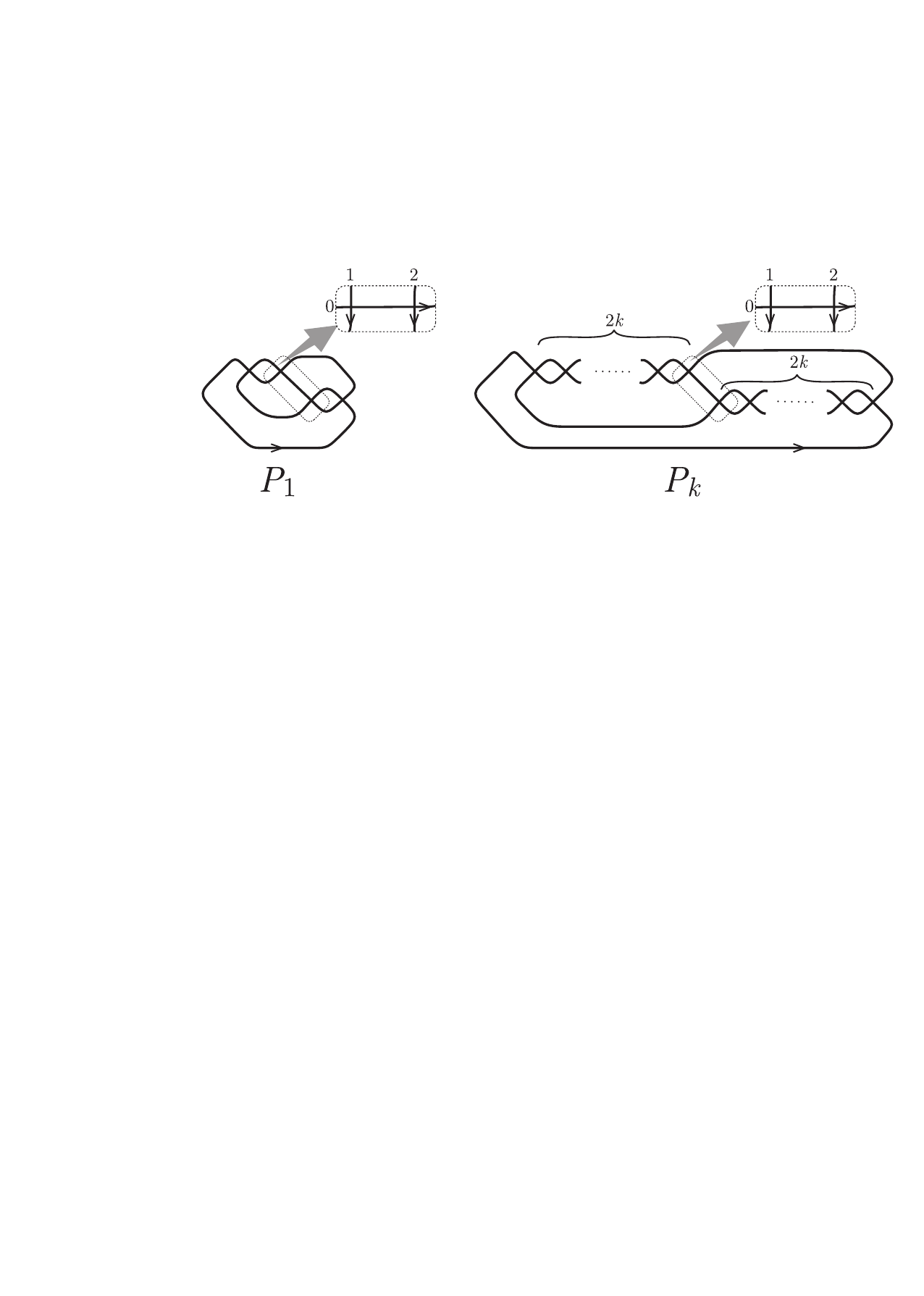}
\caption{The knot projections $P_2$ and $P_k$ ($k \geq 2$)}\label{fig:(2,4,odd)}
\end{figure}



\section*{Acknowledgments}
The authors would like to thank 
Colin Adams 
for helpful comments on the first draft of this paper.  
The second-named author has been supported in part by 
the Grant-in-Aid for Scientific Research (C), (No.~JP17K05242, No.~JP21K03220), 
Japan Society for the Promotion of Science.


\bibliographystyle{amsplain}
\bibliography{reference}

\end{document}